%% file: main.tex
\documentclass[12pt,reqno]{amsart}
\usepackage[top=1in, bottom=1in, left=1in, right=1in]{geometry}
\usepackage{tikz-cd}
\usepackage{comment}

\input{preamble}

\title{The Koszul Property for Truncations of Nonstandard Graded Polynomial Rings}
\author{Caitlin M. Davis}
\address{Department of Mathematics, University of Wisconsin-Madison, Madison, WI}
\email{cmdavis22@wisc.edu}

\author{Boyana Martinova}
\address{Department of Mathematics, University of Wisconsin-Madison, Madison, WI}
\email{martinova@wisc.edu}

\date{\today}

\begin{document}

\maketitle

\begin{abstract}
We prove that truncations of nonstandard graded polynomial rings are (nonstandard) Koszul modules in the sense of Herzog and Iyengar.  This provides an analogue of the fact that such truncations have linear resolutions in the standard graded case.
\end{abstract}

\section{Introduction}

Throughout this paper, $S$ will be a (possibly nonstandard) $\mathbb Z$-graded polynomial ring over an arbitrary field $k$. We assume that the grading is positive, in the sense that the degrees of all the variables are in $\mathbb Z_{>0}$. The main result of this article is the following theorem.

\begin{introtheorem}\label{TheoremA}
    Let $S$ be a (possibly nonstandard) $\mathbb Z$-graded polynomial ring.  Then $S_{\geq e}$ is a nonstandard Koszul module for all $e \in \mathbb Z$.
\end{introtheorem}

Nonstandard Koszulness, which is formalized in Definition \ref{NonstdKoszul}, is a homological property similar to the well-studied notion of Koszulness. Before discussing the definition of Koszulness and the generalization to the nonstandard graded case, we first give some definitions and motivation for studying truncations.

\begin{definition}
 Let $M$ be a graded $S$-module.  For $e \in \mathbb Z$, the \textit{$e^{\text{th}}$ truncation of $M$}, denoted by $M_{\geq e}$, is defined to be
    \[
    M_{\geq e} = \bigoplus_{i \geq e} M_i.
    \]
\end{definition}

Truncations can be defined similarly in the multigraded setting, at least so long as we have a positive grading in the sense of \cite[Appendix A]{BrownErman2024TateRes}; in that case, we can take the partial order $i\geq e$ to be defined by the condition $S_{i-e}\ne 0$. Maclagan and Smith introduced an alternative notion of truncations, utilizing the nef cone of the corresponding variety, in \cite{MaclaganSmith2004}. In addition, Cranton Heller introduced yet another notion of truncation in the recent preprint \cite{CrantonHeller2025}, where her definition is based on resolutions of the diagonal from \cite{HanlonHicksLazarev2024,BrownErman2024ShortProofHHL}. However, these multigraded definitions will not be relevant to the questions in this article.

Understanding homological properties of truncations of a module can have significant implications towards understanding those of the module itself.  In the standard graded setting, Eisenbud and Goto proved in \cite{EisenbudGoto1984} that there are three equivalent notions of Castelnuovo-Mumford regularity, one of which involves resolutions of truncations of a module.  Extending Eisenbud and Goto's result to multigraded regularity has been an open problem. For products of projective spaces, Bruce, Cranton Heller, and Sayrafi \cite{BruceCrantonHellerSayrafi2021} identified a characterization of multigraded regularity in terms of truncations by making an adjustment to the notion of a linear resolution.  They then leveraged this characterization to make computing multigraded regularity on products of projective spaces more manageable.  See also the introduction of \cite{CrantonHeller2025} for further motivation for studying truncations in a multigraded setting.

Turning specifically to the nonstandard $\mathbb Z$-graded situation, Brown and Erman recently investigated Betti tables of truncations of nonstandard graded modules in \cite{bePositivity2024}.  Our main result (Theorem \ref{TheoremA}) contributes to this ongoing study of truncations and gives a deeper understanding of their structure.  Namely, Brown and Erman’s result focuses on Betti numbers, while the Koszul property from Theorem \ref{TheoremA} provides a partial description of their differentials. 

Koszul algebras form an important and much-studied class of rings across many areas of algebra. Moreover, Koszul algebras form a bridge between commutative and noncommutative algebra, as much of the theory agrees in both settings. The theory originates with Priddy (\cite{Priddy1970}); see \cite{ConcaDeNegriRossi2013}, Section 2 of \cite{BeilinsonGinzburgSoergel1996}, and the references within for a thorough review. Although Koszulness is a statement about the resolution of the residue field, it is a strong property that bounds the complexity of modules over the ring. In particular, if $R$ is Koszul, then every finitely generated $R$-module has finite Castelnuovo-Mumford regularity \cite{AvramovEisenbud1992,AvramovPeeva2001}. Koszul algebras show up in many combinatorial settings as well; see \cite{AlmousaReinerSundaram2024, Coron2023, JosuatVergesNadeau2024, LaClairMastroeniMcCulloughPeeva2024, ReinerStamate2010} for some examples. Koszulness is also at the heart of fundamental topics, such as Koszul duality and the Berstein-Gel'fand-Gel'fand correspondence \cite{BersteinGelfandGelfand1978}, making it a particularly useful property to study. 

Herzog, Reiner, and Welker (\cite{HerzogReinerWelker1998}) extended the notion of Koszul algebras to the setting of local rings, and Herzog and Iyengar (\cite{HerzogIyengar2005}) then extended this to a definition of Koszul modules over local rings.  The same definition naturally applies to rings with nonstandard gradings (in fact, semigroup rings were the focus of Herzog, Reiner and Welker’s original work). Koszul modules have also been studied under the name ``modules with linear resolution," however this is only an apt name in the standard graded setting. Modules of this form are particularly nice because they are known to satisfy important homological properties in the local and standard graded settings. For example, it is immediate from the definition that such modules have regularity zero. Another nice homological property is that the Poincar\'e series of a Koszul module is a rational function \cite{HerzogIyengar2005}. Notably, Eisenbud-Goto \cite{EisenbudGoto1984} show that high degree truncations of any module over a polynomial ring are Koszul in the standard graded case.  In this article, we study the analogous question in the setting of nonstandard graded polynomial rings.

Nonstandard graded and, more generally, multigraded rings have been at the center of much active research in recent years. Notably, Benson \cite{Benson2004} generalized Castelnuovo-Mumford regularity to the nonstandard graded setting, and this played a major role in Symonds’s results on invariant theory in positive characteristic \cite{Symonds2011}. Maclagan and Smith \cite{MaclaganSmith2004} further extended this notion to the multigraded setting. These developments inspired many open questions in commutative algebra and algebraic geometry.  Questions concerning multigraded regularity have been addressed by: \cite{BotbolChardin2017, BruceCrantonHellerSayrafi2021, BruceCrantonHellerSayrafi2022, ChardinHolanda2022, ChardinNemati2020, SidmanVanTuylWang2006}.  Multigraded syzygies have also been of great interest; see \cite{BuseChardinNemati2022, beLinSyzOCurvesIWProjSp,BrownErman2024LinearStrands, bePositivity2024, BrownErman2024TateRes,  BerkeschKleinLoperYang2021, BrownSayrafi2024, Cobb2024, EisenbudErmanSchreyer2015, HaradaNowrooziVanTuyl2022, HaimanSturmfels2002, HeringSchenkSmith2006, SidmanVanTuyl2004, Yang2021}. 

Addressing (and even formulating) questions in the multigraded setting has often required reexamining and altering definitions, statements, and techniques from the standard graded setting. See, for example, \cite{BerkeschErmanSmith2017, MaclaganSmith2004, BrownErman2024TateRes}. The following definition due to Herzog-Iyengar \cite{HerzogIyengar2005} makes the necessary adjustments to the definition of Koszul modules. 

\begin{definition}(\cite[cf.~Definition~1.3]{HerzogIyengar2005})
\label{NonstdKoszul}
    Let $S$ be a nonstandard graded or multigraded $k$-algebra with maximal ideal $\mathfrak m$ and $M$ a graded $S$-module.  We say $M$ is a \textit{nonstandard Koszul module} if $\grm(M)$ has a linear resolution over the (standard graded) ring $\grm(S)$.
\end{definition} 

Nonstandard Koszul algebras can be defined in terms of the above definition as follows: a $k$-algebra $R$ is said to be a \textit{nonstandard Koszul algebra} if $k$ is a nonstandard Koszul module over $R$.

The following serves as a first example of our main result.

\begin{example}
\label{ex:IntroExample}
    Consider the nonstandard graded polynomial ring $S = k[x,y]$, where $\deg(x) = 1$ and $\deg(y) = 3$. Let $M = S_{\geq 5}$. By Theorem $\ref{TheoremA}$, $M$ is a nonstandard Koszul module, meaning that $\grm(M)$ has a linear resolution over $R = \grm(S)$. For small examples like this one, we can find the linear resolution explicitly. 

    Observe that $S_{\geq 5}$ is generated as an $S$-module by $\{x^5, x^2y, y^2\}$, and has resolution 
\[
\begin{tikzcd}
    0\arrow[r]  & S^2(-8)\arrow[r, "\varphi"] & S^2(-5) \oplus S(-6)\arrow[r] & S_{\geq 5}\arrow[r] & 0,
\end{tikzcd}
\]
where \[\varphi = \begin{bmatrix}
        y & 0\\
        -x^3 & y \\
        0 & -x^2
    \end{bmatrix}.
\]

We will denote the resolution of $S_{\geq 5}$ by $ \mathcal F_{\bullet}$. Strictly speaking, this resolution begins at $S^2(-5)\oplus S(-6)$, but it will prove useful to include $M$ when we write out the complex. Following \cite[Proposition 1.5]{HerzogIyengar2005}, we can check that $M$ is a nonstandard Koszul module by verifying that the linear part of the resolution of $M$, which we denote by $\lin(\mathcal F_{\bullet})$, is acyclic. In this case,

\[
\begin{tikzcd}
    \lin(\mathcal F_{\bullet})\colon  0 \arrow[r]& R^2(-1) \arrow[r, "\overline{\varphi}"]& R^3,
\end{tikzcd}
\]
where 
\[
\overline{\varphi} = \begin{bmatrix}
        y & 0\\
        0 & y \\
        0 & 0
    \end{bmatrix}.
\]

This complex is acyclic (as it is exact at $R(-1)^2$), so $M$ is nonstandard Koszul.  Furthermore, 
\cite[Proposition 1.5]{HerzogIyengar2005} tells us that, when $\lin(\mathcal F_{\bullet})$  is acyclic, it resolves $\grm(M)$.  For this example, we can therefore conclude that $\grm(S_{\geq 5}) \cong R/(y)\oplus R/(y) \oplus R$, and the above is its linear resolution over $\grm(S)$.
\end{example}

Proving Theorem \ref{TheoremA} amounts to showing that $\grm(S_{\geq e})$ has a linear resolution over $\grm(S)$. In Example \ref{ex:IntroExample}, we were able to realize $\grm(S_{\geq e})$ as a direct sum of modules with known linear resolutions.  In general, such a simple direct sum decomposition is not possible. Instead, our proof of Theorem \ref{TheoremA} relies on induction on the number of variables and utilizes filtrations to express $\grm(S_{\geq e})$ as an extension of $\grm(S)$-modules whose resolutions are known. Since $\grm(-)$ is not even a functor, the difficulty lies in showing that a short exact sequence of $S$-modules gives a short exact sequence of $\grm(S)$-modules. From there, we iteratively apply the Horseshoe Lemma to construct a linear resolution of $\grm(S_{\geq e})$.

Theorem \ref{TheoremA} combined with the regularity of the truncation gives strong conditions on the structure of the resolution of $S_{\geq e}$.

\begin{example}
    Let $S = k[x,y]$, where $\deg(x) = 1$ and $\deg(y) = 4$, and consider $M = S_{\geq 5}$. Let $\mathcal{F}_{\bullet}$ be the minimal free resolution of $M$ over $S$.  $M$ is generated by $\{x^5, xy, y^2\}$ as an $S$-module, so $F_0 = S(-5)^2 \oplus S(-8)$.  

    Since $M$ is nonstandard Koszul, each column of the matrix representing the map $F_1 \rightarrow F_0$ must contain an entry which is linear (i.e. a linear combination of the variables). If a column doesn't contain such an entry, then the linear part of the resolution will fail to be acyclic, which would contradict the nonstandard Koszulness of $M$. Once we know a column contains a linear entry, we know that the corresponding summand of $F_1$ can only have a twist of the form $a-b$ where $a$ is a twist appearing in $F_0$ and $b$ is a degree of a variable in $S$.  In particular, $F_1$ can only contain summands $S(-6)$, $S(-9)$, and $S(-12)$.

    On the other hand, we know that $\reg(M) = 5$. This can be obtained by examining the long exact sequence in local cohomology that arises from the short exact sequence
    \[
    0 \longrightarrow S_{\geq 5} \longrightarrow S \longrightarrow S/S_{\geq 5} \longrightarrow 0.
    \]
    See \cite[p. 501]{Symonds2011} for the definition of regularity in terms of local cohomology. In this case, the Betti table can only have $5 + 3 = 8$ rows, where $\sigma(S) = 3$ is the Symonds' constant (see \cite{Symonds2011}). Thus, the twists appearing in $F_i$ must be at most $i+8$, and the biggest summand that can appear in $F_1$ is $S(-9)$. Importantly, regularity gives only an upper bound on the degrees of the syzygies.  For example, regularity alone does not rule out direct summands such as $S(-7)$ from appearing in $F_1$. Note that, in general, $\reg(S_{\geq d}) \leq d$, but equality may not hold. However, even this bound allows us to preclude certain direct summands. 
    
    Combining these two conditions yields a strong restriction on possible summands; namely, the only possible summands that can appear in $F_1$ are $S(-6)$ and $S(-9)$. 

    To confirm, we note that the resolution of $M$ is given by
    \[\begin{tikzcd}
    0\arrow[r]  & S^2(-9)\arrow[r, "\phi"] & [2.3em] S^2(-5) \oplus S(-8)\arrow[r] & M\arrow[r] & 0,
\end{tikzcd}\]
\end{example}
where 
\[ \phi = \begin{bmatrix}
        y & 0\\
        -x^4 & y \\
        0 & -x
    \end{bmatrix}.\]

A final piece of motivation for Theorem \ref{TheoremA} comes from questions related to nonstandard graded Veronese subrings.  In the standard graded setting, it is known that Veronese subrings are Koszul; see \cite{ConcaDeNegriRossi2013} for a survey of this and related results. One proof of this fact uses the linear resolutions of truncations as a stepping stone to proving Koszulness for Veronese rings.  We conjecture that analogous results hold for any $\mathbb Z$-graded polynomial ring.

\begin{conj}[Davis-Erman-Martinova]
    If $S$ is a nonstandard graded polynomial ring and $S^{(e)}$ the $e$-th Veronese subring, then $S^{(e)}$ is a nonstandard Koszul algebra for $e$ sufficiently large.
\end{conj}

\begin{remark}
    We are grateful to Alexandra Seceleanu for pointing out an error in an earlier draft. We had previously excluded the hypothesis that $e$ be sufficiently large; however, examples such as $k[x,y,z]^{(15)}$ where $\deg(x) = 3$, $\deg(y) = 4$, and $\deg(z) = 5$ prove this to be false. To motivate this additional hypothesis, note that if $R$ is any standard graded $k$-algebra, then $R^{(e)}$ is Koszul for $e$ sufficiently large.
\end{remark}

We emphasize that Theorem \ref{TheoremA} alone does not lead to a proof of this conjecture. But it demonstrates that Koszulness results from the classical graded setting have analogues in the nonstandard graded setting, providing some motivation for the conjecture.

\subsection*{Acknowledgments}
We are tremendously grateful to Daniel Erman for his guidance, support, and feedback throughout this project. We would also like to thank Maya Banks, John Cobb, Jose Israel Rodriguez, and Aleksandra Sobieska for their mentorship and helpful comments. Multiple discussions with Srikanth Iyengar helped shape our understanding of Koszul modules and relevant results, for which we are very appreciative. We are very grateful to  Aldo Conca, Lauren Cranton Heller, David Eisenbud, Jason McCullough, Alexandra Seceleanu, Gregory Smith, and many more for their valuable insights and suggestions. Many of these conversations occurred at the Introductory Workshop in Commutative Algebra at SLMath, so we would also like to thank SLMath for hosting events like this that provide the opportunity for researchers to discuss their work. We acknowledge the support from the National Science Foundation Grant DMS-2200469, as well as that from the math departments at the University of \Hawaii \ at \Manoa \ and the University of Wisconsin-Madison. Lastly, most computations were performed with the aid of Macaulay2 \cite{M2}.

\section{Inductive Setup}

Let $S = k[x_1, x_2, \ldots, x_n]$ be a possibly nonstandard graded polynomial ring with maximal ideal $\mathfrak m = \langle x_1,\ldots,x_n \rangle$, and let $R = \grm (S)$. For fixed $j<n$, let $A = k[x_1, \ldots, x_j]$, with maximal ideal $\mathfrak n = \langle x_1, \ldots x_j \rangle$, and let $B = \gr_{\mathfrak n}(A)$.

Since we have a ring map $R \to B$, any $B$-module $N$ is naturally an $R$-module, $N_R$, which is annihilated by $\langle x_{j+1}, \ldots, x_n \rangle$. Similarly, any $A$-module $C$ is naturally an $S$-module, denoted by $C_S$. Extending modules in this way interacts with the $\grm(-)$ construction as might be expected. 

 \begin{lemma}
\label{lemma:extendinggr}
    For any $A$-module $C$, we have
    \[
    \grm(C_S) \cong (\gr_{\mathfrak n}(C))_{\grm(S)}.
    \]
\end{lemma}
\begin{proof}
    Recall from the definition of $\grm(-)$ that
    \[
    \grm(C_S) = \bigoplus_i \mathfrak m^i C_S/\mathfrak m^{i+1} C_S.
    \]
    By definition, elements in the $i$-th piece have degree $i$ in this graded module.  Similarly,
    \[
    \gr_{\mathfrak n}(C) = \bigoplus_i \mathfrak n^iC/\mathfrak n^{i+1}C,
    \]
    where again elements in the $i$-th piece have degree $i$.  Note that $\mathfrak m^i C_S = \mathfrak n^iC$ as abelian groups for all $i$, meaning that
    \[
    \bigoplus_i \mathfrak m^i C_S/\mathfrak m^{i+1} C_S = \bigoplus_i \mathfrak n^iC/\mathfrak n^{i+1}C
    \]
    as abelian groups.  Furthermore, the equalities $\mathfrak m^i C_S = \mathfrak n^iC$ respect multiplication by any variable.  Therefore, the module structures of the associated graded objects will coincide in the natural way, so that $\grm(C_S) \cong (\gr_{\mathfrak n}(C))_{\grm(S)}$ as $\grm(S)$-modules.
\end{proof}

In the proof of Theorem \ref{TheoremA}, we will induct on the number of variables, $n$.  The following proposition will allow us to construct a resolution over $\grm(S)$ from a resolution over $\grm(A)$. 

\begin{proposition}
\label{cor:inductiveStep}
    Let $C$ be any $A$-module and $C_S$ be the corresponding $S$-module. If $\gr_{\mathfrak n}(C)$ has a linear resolution over $\gr_{\mathfrak n}(A)$, then $\grm(C_S)$ has a linear resolution over $\grm(S)$. 
\end{proposition}

The following lemma is necessary for the proof of the above proposition. 

\begin{lemma}
\label{lemma:Koszulcomplex}
    Let $R$, $B$, and $N$ be as above. If $\mathcal F_\bullet$ is a resolution of $N\otimes R$ over $R$ and $\mathcal G_\bullet$ is the Koszul complex of of $\{x_{j+1}, \ldots , x_n\}$ over $R$, then the totalization of the double complex $\mathcal F_\bullet \otimes \mathcal G_\bullet$ is a resolution of $N_R$ over $R$.
\end{lemma}

\begin{proof}[Proof of Lemma \ref{lemma:Koszulcomplex}]
    Let $\mathfrak m_j = \langle x_{j+1}, \ldots , x_n \rangle$, and note that $N_{R} \cong (N \otimes_{B} R) \otimes_{R} (R / \mathfrak m_j) = \Tor_0^{R}(N \otimes_{B} R, R / \mathfrak m_j).$ Therefore, we get a complex 
    \[ 
    F_{\bullet} \otimes_R \mathcal G_{\bullet} \longrightarrow N_R \longrightarrow 0.\] 
    In order to show this complex is in fact a resolution, we must show there is no homology (equivalently, that $\Tor_i^R(N \otimes_B R, R / \mathfrak m_j) = 0$ for $i > 0$). We recall that one of the ways to compute $\Tor_i^R(N \otimes_B R, R / \mathfrak m_j)$ is as $H_i(\mathcal G_{\bullet} \otimes_R (N \otimes_B R))$. 

    In our case, $\mathcal G_{\bullet}$ is the Koszul complex of $\{x_{j+1}, \ldots , x_n\}$ over $R$, so $\mathcal G_{\bullet} \otimes_R (N \otimes_B R)$ is the Koszul complex of $\{x_{j+1}, \ldots , x_n\}$ over $N \otimes_B R$. Since $\{x_{j+1}, \ldots , x_n\}$ is a regular sequence over $N \otimes_B R$, the Koszul complex is exact. Thus, $H_i(\mathcal G_{\bullet} \otimes_R (N \otimes_B R)) = 0$ for all $i > 0$, as desired. 
\end{proof}

We can now prove Proposition \ref{cor:inductiveStep}.

\begin{proof}[Proof of Proposition \ref{cor:inductiveStep}]
Let $R = \grm(S)$ and $B = \gr_{\mathfrak n}(A)$ and $N =\gr_{\mathfrak n}(C)$, a $B$-module.  Let $\mathcal H_{\bullet}$ be the resolution of $N$ over $B$.  Since $R$ is a flat $B$-module, applying $-\otimes_B R$ to $\mathcal H_{\bullet}$ gives a resolution of $N \otimes_B R$ over $R$, which we'll denote by $\mathcal F_{\bullet}$.  Applying Lemma \ref{lemma:Koszulcomplex} yields a resolution of $N_R$ over $R$. Unraveling notation, $N_R = (\gr_{\mathfrak n}(C))_{\grm(S)}$, which Lemma \ref{lemma:extendinggr} shows is isomorphic to $\grm(C_S)$. Thus, we have a linear resolution of $\grm(C_S)$ over $R = \grm(S)$. 
 
\end{proof}

\section{Main results}

\begin{theorem}
\label{theorem:Strunction}
    Let $S = k[x_1, x_2, \ldots, x_n]$, $m = \langle x_1,\ldots,x_n \rangle$ and let $R = \grm S$. Then, the $R$-module $\grm (S_{\geq e}) $ has a linear resolution for any integer $e$. 
\end{theorem}

\begin{remark}
    The $\grm(-)$ construction does not see twists, and thus $\grm(M(a))$ is identical to $\grm(M)$ for any $M$ and any $a$.
\end{remark}

Example \ref{ex:Horseshoe} follows through all the main ideas in the proof for a simple example.  The reader may want to read this before, or in parallel with, the proof of Theorem \ref{theorem:Strunction}.

\begin{proof}
    First note that if $e \leq 0$, then $S_{\geq e} = S$, so $\grm(S_{\geq e}) = \grm(S)$, which trivially has a linear resolution over itself.  From this point on, we will assume that $e > 0$.

    We will prove by induction on the number of variables, $n$, that $\grm(S_{\geq e})$ has a linear resolution over $R$ for all $e$. Reorder the variables if needed so that $\deg(x_n) \geq \deg(x_i)$ for $i<n$. 

    As a base case for the induction, consider $S = k[x_1]$ with $\deg(x_1) = d_1$.  Every truncation of $S$ is principal and thus isomorphic to a free module.  Specifically, $S_{\geq e} = \langle x_1^i \rangle$ where $i = \lceil \frac{e}{d_1}\rceil$.  It follows that $R \longrightarrow \grm(S_{\geq e}) \longrightarrow 0$ is a linear resolution for every $e$.

    Assume the claim holds for a polynomial ring with $n-1$ variables, and let $A = k[x_1,\ldots,x_{n-1}]$, where the variables have the same degrees as they do in $S$.  We relabel $x_n$ to be $y$ so that $S = A[y]$. Let $\deg(y) = d$.

    Fix $e\ge 0$, and let $M = S_{\ge e}$.  We will express $\grm(M)$ as an extension of $R$-modules whose resolutions are known and apply the Horseshoe Lemma.  First, let $M^{(i)} = M \cap \langle y^i \rangle$. Since $S = \langle y^0 \rangle \supseteq \langle y^1 \rangle \supseteq \cdots$, we have $M = M^{(0)} \supseteq M^{(1)} \supseteq \cdots$, so the set of $M^{(i)}$'s define a filtration of $M$. 

    In the case that $i\geq\lceil \frac{e}{d}\rceil$, we have $y^i \in M$ since $d\cdot\lceil \frac{e}{d} \rceil \geq e$.  Thus, $M^{(i)} = \langle y^i \rangle$ for $i\geq\lceil \frac{e}{d}\rceil$. The least such $i$ will prove to be useful, so we fix $N = \lceil \frac{e}{d} \rceil$. Since $M^{(N)}$ is generated by a single element, it is isomorphic to $S$ (as $S$-modules).

  We claim that $M^{(i)}/M^{(i+1)}$ is isomorphic to $A_{\geq e-di} \cdot S/\langle y\rangle$.  Let $\psi$ be the composition $M^{(i)}\hookrightarrow \langle y^i\rangle \twoheadrightarrow \langle y^i\rangle/\langle y^{i+1}\rangle$.  In order to understand the filtration $M^{(i)}$, it will be convenient to realize $M$ as the following sum: 
  \[M = A_{\geq e} + A_{\geq e-d}\cdot \langle y\rangle + \ldots + A_{\geq e-(N-1)d}\cdot \langle y^{N-1}\rangle +\langle y^N\rangle.\]
  
  For $0\le i \le N$, we have $M^{(i)} = A_{\geq e-di}\cdot \langle y^i\rangle + \ldots + \langle y^N\rangle$. Then,
  \[ \Ima(\psi) = \left(A_{\geq e-di} \cdot \langle y^i\rangle + \langle y^{i+1} \rangle\right)/\langle y^{i+1}\rangle,\] 
  which is isomorphic to $A_{\geq e-di} \cdot S/\langle y \rangle = A_{\geq e-di}  \cdot A$. Note that $\Ima\psi \cong A_{\geq e-di}$ as $A$-modules. We also have $\ker \psi = M^{(i+1)}$, so $M^{(i)}/M^{(i+1)}\cong A_{\geq e-di}\cdot S/\langle y\rangle$, as claimed. Thus, for every such $i$, we have short exact sequences: 
    \[0 \longrightarrow M^{(i+1)} \longrightarrow M^{(i)} \overset{\alpha_i}{\longrightarrow} A_{\geq e-di} \cdot S/\langle y\rangle \longrightarrow 0.\]

    Fix $i <N$.  From the surjection $\alpha_i\colon M^{(i)}\rightarrow A_{\geq e-di}\cdot S/\langle y\rangle$, we get a surjection $\varphi_i\colon \mathfrak{m}^jM^{(i)}\rightarrow \mathfrak{m}^j \left(A_{\geq e-di}\cdot S/\langle y\rangle\right)$.
    For $u\in \mathfrak{m}^j$ and $v \in M^{(i)}$, $\varphi_i\colon uv \mapsto u\alpha_i(v)$.

    We want to show that $\ker\varphi_i = \mathfrak{m}^j\ker\alpha_i$.  It's enough to consider monomials in these modules. 
    If we take $u \in \mathfrak{m}^j$ and $v \in \ker\alpha_i$, then $\varphi_i(uv)= u\alpha_i(v) = 0$.  For the reverse inclusion, suppose $(\mathbf{x}^{\mathbf{a}}y^b)(\mathbf{x}^{\mathbf{s}}y^t)$ is in $\ker \varphi_i$ where $\mathbf{x}^{\mathbf{a}}y^b \in \mathfrak{m}^j$ and $\mathbf{x}^{\mathbf{s}}y^t \in M^{(i)}$. (Here  we let $\mathbf{x}^{\mathbf{a}}$ denote the monomial $x_1^{a_1}\cdots x_{n-1}^{a_{n-1}}$.) If $t > i$ then $\mathbf{x}^{\mathbf{s}}y^t \in M^{(i+1)}$, so $(\mathbf{x}^{\mathbf{a}}y^b)(\mathbf{x}^{\mathbf{s}}y^t) \in \mathfrak{m}^j\ker\alpha_i$.

    If $t=i$, then $\alpha_i(\mathbf{x}^{\mathbf{s}}y^t) \neq 0$. In this case, $(\mathbf{x}^{\mathbf{a}}y^b)(\mathbf{x}^{\mathbf{s}}y^t) \in \ker \varphi_i$ implies that $\mathbf{x}^{\mathbf{a}}y^b = 0$ in $S/\langle y \rangle$, so $b \geq 1$. Since $i<N$, $di<e$, so $y^i$ is not in $M$, and thus not in $M^{(i)}$.  Therefore, some $s_k\ge 1$.  Define $\mathbf{a'} = (a_1,\ldots, a_{k-1},a_{k}+1,a_{k+1},\ldots, a_{n-1})$ and $\mathbf{s'} = (s_1,\ldots, s_{k-1},s_{k}-1,s_{k+1},\ldots, s_{n-1})$.  Then we have $(\mathbf{x}^{\mathbf{a}}y^b)(\mathbf{x}^{\mathbf{s}}y^t)= (\mathbf{x}^{\mathbf{a'}}y^{b-1})(\mathbf{x}^{\mathbf{s'}}y^{t+1}) $, which is in $\mathfrak{m}^j\ker\alpha_i$ since $\mathbf{x}^{\mathbf{s'}}y^{t+1} \in M^{(i+1)}$.  To see this, note that $\deg(\mathbf{x}^{\mathbf{s'}}y^{t+1}) \ge \deg(\mathbf{x}^{\mathbf{s}}y^{t})$ since $d \geq \deg(x_i)$ for all $i$.

    Define $A^{(i)} = A_{\geq e - di} \cdot S/\langle y\rangle$.  The above discussion shows that, for $0 \leq i < N$, we in fact have a short exact sequence
    \[0 \longrightarrow \mathfrak{m}^jM^{(i+1)} \longrightarrow \mathfrak{m}^jM^{(i)} \overset{\varphi_i}{\longrightarrow} \mathfrak{m}^j A^{(i)} \longrightarrow 0\]
    for each $j\ge0$, where the first map is induced by inclusion.  Thus, for $j$ and $j+1$ we get a commutative diagram

\[
\begin{tikzcd}
  0 \arrow[r] & \mathfrak{m}^{j+1}M^{(i+1)}\arrow[r]\arrow[d] &  \mathfrak{m}^{j+1}M^{(i)}\arrow[r, "\varphi"]\arrow[d] &  \mathfrak{m}^{j+1} A^{(i)} \arrow[r]\arrow[d] & 0 \\
  0 \arrow[r] & \mathfrak{m}^jM^{(i+1)}\arrow[r] &  \mathfrak{m}^jM^{(i)}\arrow[r, "\varphi"] &  \mathfrak{m}^j A^{(i)} \arrow[r] & 0. \\
\end{tikzcd}
\]
Since all vertical arrows are inclusions, applying the Snake Lemma yields a short exact sequence:
\[
0 \longrightarrow (\mathfrak{m}^jM^{(i+1)}) / (\mathfrak{m}^{j+1}M^{(i+1)}) \longrightarrow (\mathfrak{m}^jM^{(i)}) / (\mathfrak{m}^{j+1}M^{(i)})\overset{\varphi_i}{\longrightarrow} (\mathfrak{m}^j A^{(i)})/(\mathfrak{m}^{j+1} A^{(i)})\longrightarrow 0.
\]

Thus, for $0 \leq i < N$, each short exact sequence induces a short exact sequence in the associated graded of the form
\[0 \longrightarrow \grm(M^{(i+1)}) \longrightarrow \grm(M^{(i)}) \longrightarrow \grm(A^{(i)}) \longrightarrow 0.\]
Note that $\grm(A^{(i)}) \cong \grm( A_{\geq e-di} \cdot S/\langle y \rangle).$

The highest $i$ we consider is $N-1$, where the short exact sequence is 
\begin{equation}\label{firstSES}
0 \longrightarrow R \longrightarrow \grm(M^{(N-1)}) \longrightarrow \grm(A_{\geq e-d(N-1)}\cdot S/\langle y \rangle) \longrightarrow 0.
\end{equation}

The inductive hypothesis gives a free resolution of $\gr_{\mathfrak n}(A_{\geq e-di})$ for all $i \leq N-1$.  We can apply Proposition \ref{cor:inductiveStep} to get a linear free resolution of $\grm(A_{\geq e-di} \cdot S/\langle y \rangle)$ over $R$, again for all $i \leq  N-1$. 

Since $R$ also has a linear free resolution (itself), applying the Horseshoe Lemma to the short exact sequence (\ref{firstSES}) yields a free resolution of $\grm(M^{(N-1)})$. In particular, each step of the free resolution of $\grm(M^{(N-1)})$ is the direct sum of the corresponding pieces of the two known resolutions. Since the resolutions of $R$ and $\grm(A_{\geq e-d(N-1)}\cdot S/\langle y\rangle)$ are both linear, and since both modules are generated in degree zero, the same will be true of $\grm(M^{(N-1)})$; see Example \ref{ex:Horseshoe} below.

Let $\mathcal{F}^{N-1} $ denote the free resolution found for $\grm(M^{(N-1)}).$ We can apply the Horseshoe Lemma to the following short exact sequence to obtain a free resolution $\mathcal{F}^{N-2}$ of $\grm(M^{(N-2)}) $
\[0 \longrightarrow \grm(M^{(N-1)}) \longrightarrow \grm(M^{(N-2)}) \longrightarrow \grm(A_{\geq e-d(N-2)} \cdot S/\langle y \rangle)\rangle \longrightarrow 0.\] 

Continuing in this way, we construct linear free resolutions $\mathcal{F}^{N-i}$ for each $\grm(M^{(N-i)})$. Since $M^{(0)} = M$, this process will yield a linear free resolution of $\grm(M)$.
\end{proof}

\begin{example}
\label{ex:Horseshoe}
    Consider $S = k[x_1, x_2, y]$ where $\deg(x_1) = 1$, $\deg(x_2) = \deg(y) = 2$, and $e = 7$. Following notation from above, $N = 4$, and we will show $\grm(M^{(3)})$ has a linear resolution over $\grm(S) = R$. 

    We have the short exact sequence
    \[0 \longrightarrow R \longrightarrow \grm(M^{(3)}) \longrightarrow \grm(A_{\geq 1} \cdot S/\langle y \rangle )\longrightarrow 0.\]

    Recall that Corollary \ref{cor:inductiveStep} gives a linear resolution of $\grm(A_{\geq 1} \cdot S/\langle y \rangle )$ over $R$. For this example, we compute the explicit resolution. The inductive hypothesis gives a resolution, which we call $\mathcal{H}_\bullet$, of $\gr_{\mathfrak n} (A_{\geq 1})$ over $\gr_{\mathfrak n}(A) = B$. Notice that $\mathcal{H}_\bullet$ is the standard graded Koszul complex on two variables over $B$
    \[0 \longrightarrow B(-1) \longrightarrow B^2 \longrightarrow  \gr_{\mathfrak n}(A_{\geq 1}) \longrightarrow 0.\]

    Following the steps outlined in Corollary \ref{cor:inductiveStep}, we apply $ - \otimes_B R$ to $\mathcal{H}_\bullet$, which results in
    \[0 \longrightarrow R(-1) \longrightarrow R^2 \longrightarrow \gr_{\mathfrak n}(A_{\geq 1}) \otimes_B R \longrightarrow 0.\]

    Applying Lemma \ref{lemma:Koszulcomplex} yields
    \[0 \longrightarrow R(-2)  \longrightarrow R(-1)^3 \longrightarrow R^2 \longrightarrow \grm(A_{\geq 1} \cdot S/\langle y \rangle )\longrightarrow 0,\] 
    which is the desired linear resolution of $\grm(A_{\geq 1} \cdot S/\langle y \rangle )$ over $R$. 

    As shown below, we can then apply the Horseshoe lemma to obtain a resolution of $\grm(M^{(3)})$.  From the resulting twists, we can see that this is a linear resolution.

\[
\begin{tikzcd}
  & & 0 \arrow[d] & 0 \arrow[d] & \\
   & & R(-2) \arrow[d] & R(-2) \arrow[d] & \\
  & \arrow[d] 0 & R(-1)^3 \arrow[d] & R(-1)^3 \arrow[d] & \\
  & R \arrow[d] & R \oplus R^2 \arrow[d] & R^2 \arrow[d] & \\
  0 \arrow[r] & R \arrow[r] \arrow[d] & \grm(M^{(3)}) \arrow[r, ] \arrow[d] & \mathfrak \grm(A_{\geq 1}\cdot S/\langle y \rangle) \ar[r] \arrow[d] & 0\\
  & 0 & 0 & 0 
\end{tikzcd}
\]

The resolution of $\grm(M^{(3)})$ above would then be plugged in to the next short exact sequence.  Applying the Horseshoe Lemma to that sequence would yield a linear resolution of $\grm(M^{(2)})$.  Continuing in this way, we will eventually obtain a linear resolution of $\grm(M^{(0)}) = \grm(M) \cong \grm(S_{\geq 7}(7))$, as claimed.
\end{example}

Although Theorem \ref{theorem:Strunction} addresses only the case of truncations of $S$, the same statement can be generalized as follows to any free $S$-module.

\begin{theorem}\label{thm:FreeModules}
    If $M$ is a free $S$-module, then $M_{\geq e}$ is nonstandard Koszul for any $e$.
\end{theorem}

\begin{proof}
    Let $M = S(a_1) \oplus \cdots \oplus S(a_j)$.  Note that we have
    \begin{align*}
    M_{\geq e} &= S(a_1)_{\geq e} \oplus \cdots \oplus S(a_j)_{\geq e}\\
    &= S_{\geq a_1 + e} \oplus \cdots \oplus S_{\geq a_j +e}.
    \end{align*}

    In order to show $M_{\geq e}$ is nonstandard Koszul, we must show $\grm(M_{\geq e})$ has a linear resolution over $\grm(S)$. Because applying $\grm(-)$ respects direct sums, we have, 
    \begin{align*}
    \grm\left(M_{\geq e}\right) &= \grm\left(S_{\geq a_1 + e} \oplus \cdots \oplus S_{\geq a_j +e}\right)  \\
    &= \grm(S_{\geq a_1 + e}) \oplus \cdots \oplus \grm(S_{\geq a_j +e}).
    \end{align*}

    Each $\grm(S_{\geq a_i +e})$ has a linear resolution over $\grm(S)$ by Theorem \ref{theorem:Strunction}. If we take the direct sum of these linear resolutions, we get a linear resolution that resolves $\grm(M_{\geq e})$, as desired.
\end{proof}

A natural next question is whether a similar statement holds for any $S$-module, $M$. The analogous statement in the standard graded setting, due to Eisenbud and Goto \cite{EisenbudGoto1984}, holds only for sufficiently high truncations.  Therefore, we can expect at most for $M_{\geq e}$ to be nonstandard Koszul for $e$ sufficiently large.  Even for large $e$, the challenge lies in understanding the structure of $\grm(M_{\geq e})$. In general, it is difficult to find a closed form for $\grm(M_{\geq e})$, so our best hope of constructing a resolution is to make alterations to a resolution of $M$. To make this concrete, if $\mathcal{F}_\bullet$ is a free resolution of $M$, truncating each step gives an acyclic complex. Theorem \ref{thm:FreeModules} gives a linear resolution for $\grm((F_i)_{\geq e})$, so we could consider the totalization of the resulting double complex. However, since $\grm(-)$ is not even a functor, applying the associated graded construction to each $(F_i)_{\geq e}$ may not yield an acyclic complex, meaning that the totalization may not be acyclic either. In short, to understand the Koszulness of truncations of arbitrary $S$-modules, new results for applying $\grm(-)$ to arbitrary free resolutions may be necessary. Alternatively, it may be possible to develop different methods for obtaining nonstandard Koszulness.

\bibliographystyle{alpha}
\bibliography{bibliography}

\end{document}

%% file: preamble.tex
\usepackage{times, amsthm, amssymb, amsmath, amsfonts, graphicx, mathrsfs}
\usepackage{mathtools}
\usepackage{tikz} 
\usetikzlibrary{arrows, cd, matrix, positioning}
\usepackage{xcolor}
\usepackage{bbm}
\usepackage{xypic}
\usepackage{mathscinet}
\usepackage{hyperref}
\usepackage{wrapfig}
\usepackage{subcaption}

\newcommand{\Manoa}{M\=anoa}

\newcommand{\Hawaii}{Hawai\kern.05em`\kern.05em\relax i}

\newtheorem{theorem}{Theorem}[section]
\newtheorem{lemma}[theorem]{Lemma}
\newtheorem{proposition}[theorem]{Proposition}

\newtheorem{introtheorem}{Theorem}

\newtheorem{conj}[theorem]{Conjecture}

\theoremstyle{definition}
\newtheorem{definition}[theorem]{Definition}
\newtheorem{example}[theorem]{Example}
 
\theoremstyle{remark}
\newtheorem{remark}[theorem]{Remark}

\DeclareMathOperator{\Ima}{Im}

\DeclareMathOperator{\lin}{lin}

\newcommand{\Tor}{\operatorname{Tor}}

\newcommand{\gr}{\operatorname{gr}}
\newcommand{\grm}{\operatorname{gr_\frakm}}

\newcommand{\reg}{\operatorname{reg}}

\newcommand{\frakm}{{\mathfrak{m}}}
